\documentclass{scrartcl}

\usepackage{amsmath,amsthm,amssymb}
\usepackage[utf8]{inputenc}
\usepackage[T1]{fontenc}

\usepackage{newunicodechar}
\newunicodechar{θ}{\theta}
\newunicodechar{σ}{\sigma}
\newunicodechar{γ}{\gamma}
\newunicodechar{α}{\alpha}
\newunicodechar{∞}{\infty}
\newunicodechar{λ}{\lambda}
\newunicodechar{ℝ}{\mathbb{R}}
\newunicodechar{ℕ}{\mathbb{N}}
\newunicodechar{κ}{\kappa}
\newunicodechar{φ}{\varphi}
\newunicodechar{δ}{\delta}
\newunicodechar{∇}{\nabla}
\newunicodechar{β}{\beta}
\newunicodechar{ε}{\varepsilon}
\newunicodechar{ρ}{\rho}
\newunicodechar{Δ}{\Delta}
\newunicodechar{μ}{\mu}
\newunicodechar{ν}{\nu}

\usepackage{xcolor}

\newcommand{\cL}{\mathcal{L}}
\newcommand{\calB}{\mathcal{B}}

\newcommand{\irn}{\int_{ℝ^n}}
\newcommand{\iX}{\int_X}
\newcommand{\cLX}[1]{L^{#1}(X)}
\newcommand{\kl}[1]{\left(#1\right)}
\newcommand{\set}[1]{\left\{#1\right\}}
\newcommand{\norm}[2][]{\|#2\|_{#1}}
\newcommand{\Lrn}[1]{L^{#1}(ℝ^n)}
\newcommand{\normlrn}[2]{\norm[\Lrn{#1}]{#2}}
\newcommand{\f}[2]{\frac{#1}{#2}}
\newcommand{\nn}{\nonumber}

\newtheorem{thm}{Theorem}[section]
\newtheorem{lemma}[thm]{Lemma}
\newtheorem{example}{Example}[section]

\usepackage{authblk}

\title{A Gagliardo--Nirenberg type inequality for rapidly decaying functions}
\author{Marek Fila\footnote{e-mail: fila@fmph.uniba.sk} }
\author{Johannes Lankeit\footnote{e-mail: jlankeit@math.upb.de}}

\affil{\footnotesize Department of Applied Mathematics and Statistics, Comenius University,\\ 
Mlynsk\'a dolina, 84248 Bratislava, Slovakia}

\begin{document}
\maketitle 
\begin{abstract}
\noindent\textbf{Abstract.} We improve the Gagliardo--Nirenberg inequality 
\[
 \|\varphi\|_{L^q(\mathbb{R}^n)} \le C \|\nabla\varphi\|_{L^r(\mathbb{R}^n)} \mathcal{L}^{-(\frac 
1q - \frac{n-r}{rn})} (\|\nabla\varphi\|_{L^r(\mathbb{R}^n)}),
\]
$r=2$, $0<q<\frac{rn}{(n-r)_+}$, $\mathcal{L}$ generalizing $\mathcal{L}(s)=\ln^{-1}\frac 2s$
for $0<s<1$, from 
[M.~Fila and M.~Winkler: A Gagliardo-Nirenberg-type inequality and its applications to decay
estimates for solutions of a degenerate parabolic equation, Adv. Math., 357
(2019), https://doi.org/10.1016/j.aim.2019.106823] for rapidly decaying 
functions ($\varphi\in W^{1,r}(\mathbb{R}^n)\setminus\{0\}$ with finite $K=\int_{\mathbb{R}^n} 
\mathcal{L}(\varphi)$) by specifying the dependence of $C$ on $K$ and by allowing arbitrary 
$r\ge1$.\\

\noindent\textbf{Mathematics Subject Classification (MSC 2010):} 35A23; 26D10\\
\noindent\textbf{Keywords:} Gagliardo--Nirenberg inequality
\end{abstract}

\section{Introduction}
Being able to control the $L^p$ norm of some function for large $p$ by an $L^q$ norm with smaller 
$q$ 
can often be useful. If one can 'invest' some boundedness information on the gradient, this 
is indeed possible, as the Gagliardo--Nirenberg inequality 
(\cite{gagliardo_proprieta,gagliardo_ulteriori_proprieta,nirenberg}) 
asserts:
\begin{equation*}
 \normlrn p{φ}\le C\normlrn q {φ} ^{1-θ}\normlrn r{∇φ}^{θ}
\end{equation*}
for every function $φ\in \Lrn q \cap W^{1,r}(ℝ^n)$ and with $θ=(\f nq -\f np)/(1+\f nq-\f 
nr)\in[0,1]$.
This inequality is frequently used in the analysis of PDEs (see e.g. \cite{brezisbrowder}, 
or, to pick some examples from Genevi\`eve Raugel's work: \cite[(2.21)]{paicu_raugel_rekalo}, 
\cite[(6.14), (8.15)]{raugel_sell}, \cite[proof of (2.32)]{hale_raugel}).
Various extensions of the classical GNI are available, for instance:
It has been studied in Besov spaces (\cite{ledoux,martin_milman}), on Riemannian manifolds 
(\cite{badr_gni_on_manifolds,ceccon_montenegro}), with a BMO term 
(\cite{dao_diaz_nguyen,riviere_strzelecki,kozono_wadade,strzelecki,mccormick_etal}), 
in Orlicz spaces 
(\cite{kalamajska_pietruskapaluba,kalamajska_pietruskapaluba_withoutdoubling,kalamajska_krbec}), for 
noninteger derivatives (\cite{morosi_pizzocchero}), with weighted (\cite{duoandikoetxea_vega}) and 
anisotropic (\cite{esfahani_anisotropic}) terms,  
and extremal functions and optimal constants have been determined 
(\cite{weinstein,delPino_dolbeault_bestconstants,agueh_plaplacian,abreu_ceccon_montenegro,
kozono_sato_wadade}). For a relation with mass transport theory see 
\cite{cordero_erausquin_etal,agueh_masstransport}.\\


Recently, in \cite{filawin}, an inequality of Gagliardo--Nirenberg type has been used to 
rather precisely 
identify temporal decay rates of solutions to the Cauchy problem for $u_t=u^pΔu$, $p\ge 1$, for 
initial data with a certain spatial decay. These rates are optimal up to sublogarithmic corrections.
In order to recall the inequality from \cite{filawin}, we introduce the following:

\textbf{Condition L.} 
We assume that $s_0>0$, $\cL\in C^0([0,∞))\cap C^1((0,s_0))$ is positive, bounded,
nondecreasing on $(0,∞)$, and satisfies the condition 
that there are $a>0$, $λ_0>0$ such that 
\begin{equation}\label{cond:H}
 \cL(s)\le (1+aλ)\cL(s^{1+λ}) \qquad \text{for all } s\in(0,s_0) \text{ and } λ\in(0,λ_0).
\end{equation}

Theorem 1.1 of \cite{filawin} reads: 
\begin{thm}\label{thm:filawin}
Let $\cL$ satisfy Condition L.
For any $n\ge 1$, $K>0$ and $q\in(0,\f{2n}{(n-2)_+})$ there is $C=C(n,q,K,\cL)>0$ such that if $φ\in 
W^{1,2}(ℝ^n)\setminus\set{0}$ is a nonnegative function satisfying $\irn \cL(φ)\le K$, then 
\begin{equation}\label{result-of-filawin}
 \norm[\Lrn q]{φ}\le C \norm[\Lrn2]{∇φ} \cL^{-(\f1q-\f{n-2}{2n})}\kl{\norm[\Lrn2]{∇φ}^2}.
\end{equation}
\end{thm}

Herein, the function $\cL$ is to be thought of as a generalization of the 
prototypical examples (cf. Examples \ref{example1}, \ref{example2}):
\begin{equation}\label{examples}
 \cL(s)=\ln^{-κ}\f Ms, \qquad \cL(s)=\ln^{-κ}\ln\f Ms, \qquad 0<s<s_0=1,\quad κ>0,\quad M>e.
\end{equation}
In these cases, $\irn \cL(φ)<\infty$ if, with some positive constants $c_0, α, β, γ$,  
\[
 φ(x) \le c_0 e^{-α|x|^{β}}, \quad \text{ or } \quad φ(x)\le c_0 \exp\kl{-α\exp(β|x|^{γ})}\qquad 
\text{for all } x\inℝ^n.
\]

This theorem apparently leaves open the question \textbf{how} $C$ \textbf{in} \eqref{result-of-filawin} 
\textbf{depends on} $K$.

We will prove: 

\begin{thm}\label{thm:gni}
Let $\cL$ satisfy Condition L.
Let $n\inℕ$, $r\ge 1$, $q\in (0,\f{rn}{(n-r)_+})$ and $\nu:=\f1q-\f{n-r}{rn}$. Then for every $ε>0$ 
there is $C=C(ε,n,q,r,\cL)>0$ 
such that 
\[
 \norm[\Lrn q]{φ} \le C (1+(1+K^{ε})K^{\nu}) \norm[\Lrn r]{∇φ} \cL^{-\nu}\kl{\norm[\Lrn r]{∇φ}^2},
\]
where $K=\irn \cL(φ)$, holds for every nonnegative $φ\in W^{1,r}(ℝ^n)\setminus\set{0}$ with $K<\infty$. 
\end{thm}

The proof is analogous to that in \cite{filawin} with more careful tracking of constants and 
occasional modifications. 
Moreover, Theorem \ref{thm:gni} generalizes Theorem \ref{thm:filawin} from the case $r=2$ to 
more general values of $r$.

\section{Consequences of Condition L}\label{sec:condition_and_consequences}

As already hinted at in \eqref{examples}, typical cases are given by the following examples. 
\begin{example}\label{example1}
 Let $κ>0$, $M\ge 2$ and 
 \[
  \cL(s):=\begin{cases}
          0, &s=0,\\
          \ln^{-κ}\f Ms, & s\in(0,\f M2),\\
          \ln^{-κ}2, &s\ge \f M2.
         \end{cases}
 \]
\end{example}

\begin{example}\label{example2}
 Let $κ>0$, $M>e$ and $s_0\in[1,\f M e)$ and 
 \[
  \cL(s):=\begin{cases}
          0,&s=0,\\
          \ln^{-κ}\ln\f Ms,&s\in(0,s_0),\\
          \ln^{-κ}\ln \f{M}{s_0}, &s\ge s_0.
         \end{cases}
 \]
\end{example}
For proofs of Condition L being fulfilled cf. \cite[Lemmata 3.9 and 3.11]{filawin}. 

In the following lemmata, we collect some properties that follow from the above condition: 

\begin{lemma}\label{lem:war2.1}
 Assume that $s_0\in(0,1)$ and that $\cL\in C^0([0,s_0))\cap C^1((0,s_0))$ is positive and 
nondecreasing on $(0,s_0)$ and such that \eqref{cond:H} is valid with some $a>0$ and $λ_0>0$. Then 
\[
 \f{s\cL'(s)}{\cL(s)} \le \f{a}{\ln\f1s}\qquad \text{for all } s\in(0,s_0)
\]
and in particular 
\[
 \f{s\cL'(s)}{\cL(s)} \to 0 \qquad \text{as } s\searrow 0.
\]
\end{lemma}
\begin{proof}
 This is \cite[Lemma 2.1]{filawin}. 
\end{proof}

\begin{lemma}\label{lem:s1c1}
 Given $δ\in(0,1]$, $\cL$ satisfying the assumptions of Lemma \ref{lem:war2.1}, $q_*>0$, $q>0$, 
$r>0$ there are 
$s_1\in(0,δ)\cap (0,s_0)$ and $c_1=c_1(\cL,q_*,q,r)>0$ such that 
\begin{equation}\label{estLfrombelow}
 \cL(s) \ge c_1 s^{\f{q_*}2} \qquad \text{for all } s\in(0,s_1)
\end{equation}
and 
\[
 c_1 = s_1^{-\f{q_*}2}\cL(s_1)
\]
and 
\begin{equation}\label{decreasing}
 s\mapsto s^{-t}\cL(s)\quad \text{ is decreasing on } (0,s_1)
\end{equation}
for each $t\in \{q,r\}$.
\end{lemma}
\begin{proof}
 Lemma \ref{lem:war2.1} guarantees the existence of $s_1\in(0,δ)\cap(0,s_0)$ such that 
$\f{s\cL'(s)}{\cL(s)}\le 
\min\set{q,r,\f{q_*}2}$ for all $s\in(0,s_1)$, meaning that \eqref{decreasing} holds according to 
\[
 \f{d}{ds} \kl{s^{-t}\cL(s)} = s^{-t-1}\kl{s\cL'(s)-t\cL(s)}\le 0 \qquad \text{for all } 
s\in(0,s_1),\; t\in\set{q,r,\f{q_*}2}.
\]
In particular, for every $s\in(0,s_1)$, 
\[
 s^{-\f{q_*}2}\cL(s) \ge s_1^{-\f{q_*}2}\cL(s_1) =: c_1,
\]
which proves \eqref{estLfrombelow}.
\end{proof}

\begin{lemma}\label{was:l2.2}
 Let $s_0\in(0,1)$ and $\cL\in C^0([0,s_0))\cap C^1((0,s_0))$ be positive and nondecreasing on  
$(0,s_0)$ and such that \eqref{cond:H} holds. Then for any $d\in (0,1]$ with 
\[
 C:= d^{a(\ln\f1{s_0})^{-1}}
\]
the inequality 
\[
 \cL(ds)\ge C\cL(s) 
\]
holds for all $s\in(0,s_0)$. 
\end{lemma}
\begin{proof}
 \cite[Lemma 2.2]{filawin}
\end{proof}

\section{An interpolation lemma in Lebesgue spaces}
At the core of the proof of Theorem \ref{thm:gni} lies the following interpolation result (see 
\cite[Lemma 2.3]{filawin}). Note that for the moment we do not assume boundedness of $\cL$.
\begin{lemma}\label{lem:interpolation-Lunbounded}
Assume that $s_0\in(0,1)$ and $\cL\in C^0([0,∞))\cap C^1((0,s_0))$ is nonnegative, nondecreasing and 
such that \eqref{cond:H} 
holds. Then for any choice of $n\ge 1$, $q_*>0$, $q\in(0,q_*)$ and $ε>0$ one can find 
$C=C(\cL,n,q,q_*,ε)>0$ with the property that the inequality 
\begin{equation}\label{interpolationineq-Lunbounded}
 \norm[\Lrn{q}]{φ}\le C 
\norm[\Lrn{q_*}]{φ}K^{\f1q-\f1{q_*}} \kl{1+K^{ε}}\kl{
1+\cL^ {-(\f1q-\f1{q_*})}(\norm[\Lrn { q_* }]{φ}^2)},
\end{equation}
where $K=\irn\cL(φ)$, 
holds for every nonnegative $φ\in \Lrn{q_*}\setminus\set0$.
\end{lemma}

\begin{proof}
 We choose $δ\in(0,1)$ such that (with $a$ from conditon \eqref{cond:H} on $\cL$)
 \begin{equation}\label{delta}
  -\kl{\f1q-\f1{q_*}}\f{2a}{q_*\ln δ} < ε
 \end{equation}
 and with $ε_1:=εq_*$ we define $r:=ε_1(\f1q-\f{1-ε_1}{q_*})^{-1}$. To these choices, we apply 
Lemma \ref{lem:s1c1} and obtain $s_1\in(0,δ)$ such that \eqref{decreasing} holds for each $t\in 
\{q,r\}$ and such that (by \eqref{delta})
\begin{equation}\label{eq:s1}
 \kl{\f1q-\f1{q_*}}\kl{1-\f{2a}{q_*\ln s_1}} < \f1q-\f1{q_*}+ε
\end{equation}
as well as $c_1=s_1^{-\f{q_*}2}\cL(s_1)>0$. 

We moreover introduce 
\[
 s_2:=\min\set{s_1,\inf\set{ 
x\cL^{\f1{q_*}}(c_1^{\f2{q_*}}x^2) \mid x\in (0,c_1^{-\f1{q_*}}s_1)} , c_1^{-\f1{q_*}} 
s_1 \cL^{\f1{q_*}}(s_1^2)}.
\]

Let $φ\in L^{q_*}(ℝ^n)$ be nonnegative and such that $φ\not\equiv0$.

We abbreviate 
\[
 K:=\irn \cL(φ). 
\]

If $K\ge c_1$, we let $d:=\kl{\f{c_1}K}^{\f2{q_*}}$ 
and 
$c_2(K):=\kl{\f{c_1}K}^{\f2{q_*}a(\ln \f1{s_1})^{-1}}$.

If $K<c_1$, we let $d:=1$ 
and $c_2:=1$. 

Then, in both of these cases,  
\begin{equation}\label{c2}
 \cL(ds)\ge c_2\cL(s) \qquad \text{for all } s\in(0,s_1)
\end{equation}
as follows from Lemma \ref{was:l2.2}.

Moreover, the definition of $s_2$ ensures that 
\begin{equation}\label{s2le}
 s_2\le K^{-\f1{q_*}} s_1 \cL^{\f1{q_*}}(ds_1^2).
\end{equation}

We let 
\[
 B:=K^{1-\f{q}{q_*}}\norm[\Lrn{q_*}]{φ}^{-(q_*-q)}\cL^{-\f{q_*-q}{q_*}}\kl{d\norm[\Lrn{q_*}]{φ}^2}
\]
and treat the cases $B^{-\f1{q_*-q}}\ge s_2$ (Case I) and $B^{-\f1{q_*-q}}< s_2$ (Case II) 
separately. 

\textbf{Case I}:
We begin the proof for the case of $B^{-\f1{q_*-q}}\ge s_2$ by observing that 
\begin{equation}\label{splitthenorm}
 \norm[\Lrn q]{φ}\le \norm[L^q(\set{φ\ge s_2})]{φ}+\norm[L^q(\set{φ< s_2})]{φ}. 
\end{equation}
Since $\cL$ was assumed monotone, 
\[
 K=\irn \cL(φ) \ge \cL(s_2)\left\lvert\set{φ\ge s_2}\right\rvert,
\]
and thus by Hölder's inequality 
\[
 \int_{\set{φ\ge s_2}} φ^q \le \kl{\int_{\set{φ\ge s_2}} φ^{q_*}}^{\f{q}{q_*}}\left\lvert\set{φ\ge 
s_2}\right\rvert^{\f{q_*-q}{q_*}} \le \kl{\irn 
φ^{q_*}}^{\f{q}{q_*}}\kl{\f{K}{\cL(s_2)}}^{\f{q_*-q}{q_*}}.
\]
We can therefore estimate the first summand on the right of \eqref{splitthenorm} by 
\begin{equation}\label{eq:firstsummand}
 \norm[L^q(\set{φ\ge s_2})]{φ}\le c_3 K^{\f1q-\f1{q_*}} \norm[\Lrn{q_*}]{φ}, 
\end{equation}
where
\[
 c_3 := \cL(s_2)^{-\kl{\f1q - \f1{q_*}}}.
\]

In dealing with the last term in \eqref{splitthenorm}, $\norm[L^q(X)]{φ}$ for $X:= \set{φ<s_2}$, we 
again separate different cases.

\textbf{Case I a)}: $\norm[\Lrn{q_*}]{φ}/\sqrt{s_2} > 1$.\\
Here we will rely on the fact that by \eqref{decreasing} we have 
\[
 \f{φ^r(x)}{\cL(φ(x))} \le c_4 := \f{s_2^r}{\cL(s_2)} \quad \text{for every } x\in X = \set{φ<s_2}. 
\]
From Hölder's inequality and this estimate, namely, we can infer that 
\begin{align}
 \norm[\cLX{q}]{φ} &\le \norm[\cLX{\kl{\f1q-\f{1-ε_1}{q_*}}^{-1}\hspace{-4mm}}]{φ^{ε_1}} 
\norm[\cLX{\f{q_*}{1-ε_1
}}]{φ^{1-ε_1}}\nn\\
&= \kl{\iX φ^{ε_1(\f1q-\f{1-ε_1}{q_*})^{-1}} }^{\f1q-\f{1-ε_1}{q_*}} \kl{\iX 
φ^{q_*}}^{\f{1-ε_1}{q_*}}\nn\\
&\le \kl{\iX φ^{r} }^{\f1q-\f{1-ε_1}{q_*}} \norm[\Lrn 
{q_*}]{φ}^{1-ε_1}\nn\\
&\le \kl{\iX φ^{r} }^{\f1q-\f{1-ε_1}{q_*}} s_2^{-\f{ε_1}2} \norm[\Lrn 
{q_*}]{φ}\nn\\
&\le \kl{c_4\irn \cL(φ) }^{\f1q-\f{1-ε_1}{q_*}} s_2^{-\f{ε_1}2} \norm[\Lrn 
{q_*}]{φ}.\nn
\end{align}
and taken together with \eqref{splitthenorm} and \eqref{eq:firstsummand}, this means 
\begin{equation}\label{end-of-case-I-a}
\norm[\Lrn q]{φ}\le c_3 K^{\f1q-\f1{q_*}} \norm[\Lrn{q_*}]{φ}+c_5 K^{\f1q-\f{1-ε_1}{q_*}} 
 \norm[\Lrn {q_*}]{φ},
\end{equation}
where
\[
 c_5=s_2^{-\f{ε_1}2}c_4^{\f1q-\f{1-ε_1}{q_*}}.
\]

\textbf{Case I b)} If, on the other hand, $\norm[\Lrn{q_*}]{φ}\le \sqrt{s_2}$, then the assumption 
on $B$ implies 
\[
  K^{-\f1{q_*}} \norm[\Lrn{q_*}]{φ}\cL^{\f1{q_*}}\kl{d\norm[\Lrn{q_*}]{φ}^2} \ge s_2
\]
and thereby 
\begin{align*}
K^{-\f1{q_*}} &\norm[\Lrn{q_*}]{φ}\cL^{-(\f1q-\f1{q_*})}(\norm[\Lrn{q_*}]{φ}^2) \\
 &= K^{-\f1{q_*}} \norm[\Lrn{q_*}]{φ} \cL^{\f1{q_*}}\kl{\norm[\Lrn{q_*}]{φ}^2} 
\cL^{-\f1q}\kl{\norm[\Lrn{q_*}]{φ}^2} \\
&\ge K^{-\f1{q_*}} \norm[\Lrn{q_*}]{φ} \cL^{\f1{q_*}}\kl{d\norm[\Lrn{q_*}]{φ}^2} 
\cL^{-\f1q}\kl{\norm[\Lrn{q_*}]{φ}^2}\\
&\ge s_2 \cL^{-\f1q}\kl{\norm[\Lrn{q_*}]{φ}^2} \\
&\ge s_2\cL^{-\f1q}(s_2) \\
&=: c_6^{\f1q}.
\end{align*}
Hence 
\begin{equation}\label{estim:c5K}
 c_6^{\f1q}K^{\f1q} \le K^{\f1q-\f1{q_*}} 
\norm[\Lrn{q_*}]{φ}\cL^{-(\f1q-\f1{q_*})}\kl{\norm[\Lrn{q_*}]{φ}^2}.
\end{equation}

Similarly to the reasoning in case I a), the last summand in \eqref{splitthenorm} can be estimated 
by taking into account the monotonicity property in \eqref{decreasing} for $q$ and that $s_2\le 
s_1$, as these entail 
\[
 \f{φ^q(x)}{\cL(φ(x))} \le \f{s_2^q}{\cL(s_2)}=c_6 \quad \text{for every } x \in X
\]
and hence 
\[
 \int_{\set{φ<s_2}} φ^q \le c_6 \int_{\set{φ<s_2}} \cL(φ) \le c_6 K.
\]
In conclusion, \eqref{splitthenorm} thus is turned into 
\[
 \norm[\Lrn q]{φ}\le c_3K^{\f1q-\f1{q_*}} \norm[\Lrn{q_*}]{φ}+\kl{c_6K}^{\f1q}, 
\]
i.e. due to \eqref{estim:c5K}, 
\begin{equation}\label{end-of-case-I-b}
 \norm[\Lrn q]{φ}\le c_3K^{\f1q-\f1{q_*}}\norm[\Lrn{q_*}]{φ}+K^{\f1q-\f1{q_*}} 
\norm[\Lrn{q_*}]{φ}\cL^{-(\f1q-\f1{q_*})}\kl{\norm[\Lrn{q_*}]{φ}^2}. 
\end{equation}

\textbf{Conclusion of Case I.} 
Combining \eqref{end-of-case-I-a} and \eqref{end-of-case-I-b} we obtain 

\begin{equation}\label{end-of-case-I}
 \norm[\Lrn q]{φ}\le \kl{c_3 + c_5 K^{ε} + 
\cL^{-(\f1q-\f1{q_*})}\kl{\norm[\Lrn{q_*}]{φ}^2}} \norm[\Lrn{q_*}]{φ} K^{\f1q-\f1{q_*}}. 
\end{equation}

\textbf{Case II. }
The proof for the case 
\begin{equation}\label{caseII}
 B^{-\f1{q_*-q}}< s_2
\end{equation}
very closely follows that of \cite[Lemma 2.3]{filawin}, which already led to an estimate without 
nonexplicit constants: 
Since 
\[
 s_2>B^{-\f1{q_*-q}} =
K^{-\f{1}{q_*}}\norm[\Lrn{q_*}]{φ}\cL^{\f{1}{q_*}}\kl{d\norm[\Lrn{q_*}]{φ}^2}
\]
 and $s\mapsto s\cL^{\f1{q_*}}(ds^2)$ is nondecreasing (because $\cL$ is nondecreasing), \eqref{s2le} 
ensures that necessarily 
$\norm[\Lrn{q_*}]{φ}< s_1$.

For every $z\ge B^{-\f1{q_*-q}}$, we have $z^q\le Bz^{q_*}$, whereas for $z\in(0,B^{-\f1{q_*-q}})$ 
by \eqref{caseII} apparently $z<s_2\le s_1$ and hence \eqref{decreasing} entails 
\[
 \f{z^q}{\cL(z)}\le \f{(B^{-\f1{q_*-q}})^q}{\cL(B^{-\f1{q_*-q}})} =: \calB, 
\]
so that 
\[
 φ^q(x)\le Bφ^{q_*}(x)+\calB\cL(φ(x)) \qquad \text{for all } x\in ℝ^n,
\]
and, accordingly, 
\begin{equation}\label{integralsum}
 \irn φ^q \le B\irn φ^{q_*} +\calB K.
\end{equation}
In order to control the last term therein, we first observe that by \eqref{estLfrombelow} 
\begin{align*}
 B^{-\f1{q_*-q}} &= K^{-\f1{q_*}}\norm[\Lrn{q_*}]{φ}\cL^{\f1{q_*}}(d\norm[\Lrn{q_*}]{φ}^2) \ge 
K^{-\f1{q_*}}\norm[\Lrn{q_*}]{φ} c_1^{\f1{q_*}}\sqrt{d}\norm[\Lrn{q_*}]{φ} \\
&= 
\kl{\f{c_1}K}^{\f1q_*} \sqrt{d} \norm[\Lrn{q_*}]{φ}^2 \ge d\norm[\Lrn{q_*}]{φ}^2
\end{align*}
according to the definitions of $B$ and $d$. By the monotonicity of $\cL$, we therefore have 
\[
 \f{B}{\calB K} 
 =\f1K B^{\f{q_*}{q_*-q}} \cL(B^{-\f1{q_*-q}}) = \norm[\Lrn{q_*}]{φ}^{-q_*} 
\f{\cL(B^{-\f1{q_*-q}})}{\cL(d\norm[\Lrn{q_*}]{φ}^2)} \ge \norm[\Lrn{q_*}]{φ}^{-q_*},
\]
which turns \eqref{integralsum} into 
\[
 \irn φ^q \le 2B\irn φ^{q_*}
 =2K^{1-\f{q}{q_*}}\norm[\Lrn{q_*}]{φ}^{q}\cL^{-\f{q_*-q}{q_*}}\kl{d\norm[\Lrn{q_*}]{φ}^2}.
\]
Taking the $q$th root and employing \eqref{c2} we obtain
\[
 \norm[\Lrn q]{φ}\le 2^{\f1q} K^{\f1q-\f 1{q_*}} \norm[\Lrn{q_*}]{φ} 
c_2^{-(\f1q-\f1{q_*})} \cL^{-(\f1q-\f1{q_*})}\kl{\norm[\Lrn{q_*}]{φ}^2}.
\]
If we insert the definition of $c_2$, we see that either 
\begin{equation}\label{end-of-case-II-part-i}
 \norm[\Lrn q]{φ}\le 2^{\f1q} K^{\f1q-\f 1{q_*}} \norm[\Lrn{q_*}]{φ} 
\cL^{-(\f1q-\f1{q_*})}\kl{\norm[\Lrn{q_*}]{φ}^2}
\end{equation}
(namely, if $K<c_1$) or (if $K\ge c_1$), by \eqref{eq:s1}, 
\begin{align}\label{end-of-case-II-part-ii}
 \norm[\Lrn q]{φ} & \le 2^{\f1q} K^{\f1q-\f 1{q_*}} \norm[\Lrn{q_*}]{φ} 
\kl{\f{c_1}{K}}^{-(\f1q-\f1{q_*})\f2{q_*}a(\ln\f1{s_1})^{-1}} 
\cL^{-(\f1q-\f1{q_*})}\kl{\norm[\Lrn{q_*}]{φ}^2}\nn\\
& \le 2^{\f1q} K^{\f1q-\f 1{q_*}} \norm[\Lrn{q_*}]{φ} 
\kl{\f{c_1}{K}}^{-(\f1q-\f1{q_*})\f2{q_*}a(\ln\f1{s_1})^{-1}} 
\cL^{-(\f1q-\f1{q_*})}\kl{\norm[\Lrn{q_*}]{φ}^2}\nn\\
&\le c_7 K^{\f1q-\f 1{q_*}} \norm[\Lrn{q_*}]{φ} 
(1+K^{ε}) 
\cL^{-(\f1q-\f1{q_*})}\kl{\norm[\Lrn{q_*}]{φ}^2},
\end{align}
with 
\[
 c_7:=2^{\f1q}c_1^{-(\f1q-\f1{q_*})\f2{q_*}a(\ln\f1{s_1})^{-1}}. 
\]

\textbf{Conclusion of Cases I and II.} 
Finally, we summarize \eqref{end-of-case-I}, \eqref{end-of-case-II-part-i} and 
\eqref{end-of-case-II-part-ii}: 
\begin{align*}
 \norm[\Lrn q]{φ}\le \max  \bigg\{ &
 \kl{c_3 + c_5 K^{ε} + 
\cL^{-(\f1q-\f1{q_*})}\kl{\norm[\Lrn{q_*}]{φ}^2}} \norm[\Lrn{q_*}]{φ} K^{\f1q-\f1{q_*}}
,
  \\& 2^{\f1q} K^{\f1q-\f 1{q_*}} \norm[\Lrn{q_*}]{φ} 
\cL^{-(\f1q-\f1{q_*})}\kl{\norm[\Lrn{q_*}]{φ}^2},\\
&c_7 K^{\f1q-\f 1{q_*}} \norm[\Lrn{q_*}]{φ} 
(1+K^{ε}) 
\cL^{-(\f1q-\f1{q_*})}\kl{\norm[\Lrn{q_*}]{φ}^2}
 \bigg\},
\end{align*}
which for $C:= \max \set{2^{\f1q},c_3,c_5,c_7}$ turns into \eqref{interpolationineq-Lunbounded}.
\end{proof}

\begin{lemma}\label{lem:interpolation-Lbounded}
Assume that $\cL$ satisfies Condition L. Then for any choice of $n\ge 1$, $q_*>0$, $q\in(0,q_*)$ 
and $ε>0$ one can find 
$C=C(\cL,n,q,q_*,ε)>0$ with the property that the inequality 
\begin{equation}\label{interpolationineq-Lbounded}
 \norm[\Lrn{q}]{φ}\le C 
\norm[\Lrn{q_*}]{φ}K^{\f1q-\f1{q_*}} \kl{1+K^{ε}}\cL^ {-(\f1q-\f1{q_*})}(\norm[\Lrn { q_* }]{φ}^2),
\end{equation}
where $K=\irn\cL(φ)$, 
holds for every nonnegative $φ\in \Lrn{q_*}\setminus\set0$.
\end{lemma}
\begin{proof}
 As $\cL$ is bounded, there is $c>0$ such that $1\le c\cL^ {-(\f1q-\f1{q_*})}(s)$ for all $s>0$, 
and Lemma \ref{interpolationineq-Lunbounded} immediately implies \eqref{interpolationineq-Lbounded}.
\end{proof}

\section{Proof of the Gagliardo--Nirenberg type inequality}
\begin{proof}[Proof of Theorem \ref{thm:gni}]

Given $q\in(0,\f{rn}{(n-r)_+})$ we fix $q_*\ge 1$ such that $q_*>q$ and $q_*< \f{rn}{(n-r)_+}$. 
We let $θ=\f{\f nq-\f n{q_*}}{1+\f nq - \f nr}\in(0,1]$ and $γ=\f1q-\f1{q_*}$. We choose 
$β\in(0,1)$ and $ε_1>0$ such that 
\[
 \f{γ}{θ}\kl{\f{1}{β}-1}+\f{ε_1}{βθ}= ε
\]

From Lemma \ref{lem:interpolation-Lbounded} we therefore obtain $c_0>0$ such that 
\begin{equation}\label{previouslemma}
 \norm[\Lrn{q}]{φ}\le c_0  
\norm[\Lrn{q_*}]{φ}K^{γ} \kl{1+K^{ε_1}}
\cL^ {-γ}(\norm[\Lrn { q_* }]{φ}^2)
\end{equation}
holds with $γ=\f1q-\f1{q_*}$ for every $φ\in \Lrn{q_*}\setminus\set0$. 

We choose $s_3\in(0,1)$ so small that (in accordance with Lemma \ref{lem:war2.1}) 
\begin{equation}\label{rhoprimenonneg}
 \f{s\cL'(s)}{\cL(s)}\le \f1{2γ} \qquad \text{on } (0,s_3)
\end{equation}
and that, due to \eqref{estLfrombelow}, with a suitable $c_1>0$ and 
$μ:=\f{1-β}{2γ}>0$ we have 
\begin{equation}\label{oneoverLlesomething}
 \cL^{-1}(s)\le c_1s^{-μ} \qquad \text{for all } s\in(0,s_3). 
\end{equation}

We let $c_2\ge 1$ be a constant from the classical Gagliardo--Nirenberg inequality 
\begin{equation}\label{classGNI}
 \norm[\Lrn{q_*}]{φ}\le c_2 \norm[\Lrn r]{∇φ}^{θ}\norm[\Lrn q]{φ}^{1-θ} 
\end{equation}
with $θ=\f{\f nq-\f n{q_*}}{1+\f nq - \f nr}\in(0,1]$.

We moreover introduce 
\begin{equation}\label{defrho}
 ρ(σ):=σ\cL^{-γ}(σ^2),\qquad\qquad σ>0,
\end{equation}
and note that $ρ'(σ)=\cL^{-γ}(σ)\kl{1-2γ\f{σ^2\cL'(σ^2)}{\cL(σ^2)}}\ge 0$ for every 
$σ\in (0,\sqrt{s_3})$ by \eqref{rhoprimenonneg}. 

\textbf{Case I: $\norm[\Lrn q]{φ}>\norm[\Lrn r]{∇φ}$}.\\

\textbf{Case Ia: $c_2\norm[\Lrn r]{∇φ}^{θ}\norm[\Lrn q]{φ}^{1-θ}<\sqrt{s_3}$}.\\ 
If we again abbreviate $K:=\irn \cL(φ)$ and write \eqref{previouslemma} in terms of $ρ$, we obtain 
\begin{align*}
 \norm[\Lrn q]{φ} \le c_0 K^{γ}(1+K^{ε_1})ρ(\norm[\Lrn{q_*}]{φ})
\end{align*}
so that \eqref{classGNI} and monotonicity of $ρ$ on $(0,\sqrt{s_3})$ imply 
\begin{align*}
 \norm[\Lrn q]{φ} &\le c_0 K^{γ}(1+K^{ε_1})ρ(c_2 \norm[\Lrn r]{∇φ}^{θ}\norm[\Lrn q]{φ}^{1-θ} )\\
 &= c_0 K^{γ}(1+K^{ε_1})c_2 \norm[\Lrn r]{∇φ}^{θ}\norm[\Lrn q]{φ}^{1-θ} \cL^{-γ}(c_2^2 
\norm[\Lrn r]{∇φ}^{2θ}\norm[\Lrn q]{φ}^{2(1-θ)})
\end{align*}
and thus after division by $\norm[\Lrn q]{φ}^{1-θ}$ and taking the $θ$th root
\begin{align}\label{gni-case-Ia}
 \norm[\Lrn q]{φ} 
&\le \kl{c_0 K^{γ}(1+K^{ε_1})c_2} ^{\f1{θ}} \norm[\Lrn r]{∇φ}\cL^{-\f{γ}{θ}}(
\norm[\Lrn r]{∇φ}^{2}),
\end{align}
because $\cL$ is monotone, $c_2\ge 1$ and $\norm[\Lrn q]{φ}>\norm[\Lrn r]{∇φ}$. 


\textbf{Case Ib: $c_2\norm[\Lrn r]{∇φ}^{θ}\norm[\Lrn q]{φ}^{1-θ}\ge \sqrt{s_3}$}.\\

Here we let 
\begin{equation}\label{defsigma1}
 σ_1:=\inf ρ^{-1}\kl{\left\{\f{\sqrt{s_3}}{c_2C_*}\right\} },
\end{equation}
where $C_*:= \kl{c_0 K^{γ} \kl{1+K^{ε_1}}}$. 

From $\norm[\Lrn q]{φ}>\norm[\Lrn r]{∇φ}$ and $c_2\norm[\Lrn r]{∇φ}^{θ}\norm[\Lrn q]{φ}^{1-θ}\ge 
\sqrt{s_3}$ we conclude that 
\[
 \norm[\Lrn q]{φ}\ge \f{\sqrt{s_3}}{c_2}
\]
and since $\norm[\Lrn q]{φ}\le C_*ρ(\norm[\Lrn {q_*}]{φ})$ by \eqref{previouslemma}, this 
shows that 
\[
 ρ(\norm[\Lrn {q_*}]{φ}) \ge \f{\sqrt{s_3}}{c_2C_*} 
\]
and hence 
\[
 \norm[\Lrn{q_*}]{φ}\ge \inf ρ^{-1}\kl{\left\{ρ(\norm[\Lrn {q_*}]{φ}\right\}}\ge σ_1,
\]
because $ρ$ with $ρ(0)=0$ is continuous 
and hence $\inf ρ^{-1}(\{x\})\ge \inf ρ^{-1}(\{y\})$ for every $x,y\in(0,∞)$ with $x\ge y$.
Thus due to \eqref{previouslemma}, the monotonicity of $\cL$, \eqref{defrho}, \eqref{classGNI}, 
\eqref{defsigma1}, 
\begin{align*}
 \norm[\Lrn{q}]{φ}&\le c_0 \norm[\Lrn{q_*}]{φ}K^{γ} \kl{1+K^{ε_1}}\cL^ {-γ}(σ_1^2)\\
&\le c_0c_2 \norm[\Lrn{q}]{φ}^{1-θ}\norm[\Lrn r]{∇φ}^{θ} K^{γ} 
\kl{1+K^{ε_1}}\f{ρ(σ_1)}{σ_1}\\
&= c_0c_2 \norm[\Lrn{q}]{φ}^{1-θ}\norm[\Lrn r]{∇φ}^{θ} K^{γ} 
\kl{1+K^{ε_1}}\f{\sqrt{s_3}}{σ_1 c_2C_*}\\
&= c_0 \norm[\Lrn{q}]{φ}^{1-θ}\norm[\Lrn r]{∇φ}^{θ} K^{γ} 
\kl{1+K^{ε_1}}\cdot \f{\sqrt{s_3}}{σ_1 c_0 K^{γ} \kl{1+K^{ε_1}}}\\
&= \norm[\Lrn{q}]{φ}^{1-θ}\norm[\Lrn r]{∇φ}^{θ} 
\f{\sqrt{s_3}}{σ_1},
\end{align*}
i.e. 
\begin{equation}\label{lesqrts3sigma1}
 \norm[\Lrn{q}]{φ}\le \norm[\Lrn r]{∇φ}
\kl{\f{\sqrt{s_3}}{σ_1}}^{\f1{θ}}.
\end{equation}

For $σ\ge \sqrt{s_3}$ we have $ρ(σ)=σ\cL^{-γ}(σ^2) \le σ \cL^{-γ}(s_3)$, whereas for 
$σ<\sqrt{s_3}\le 1$ 
using \eqref{oneoverLlesomething} 
we see that 
\[
 ρ(σ)=σ\cL^{-γ}(σ^2) 
 \le  c_1 σ^{1-2μγ} = c_3σ^{β}.
\]

Therefore (and due to $ρ(σ_1)=\f{\sqrt{s_3}}{c_2C_*}$)
\[
 \f1{σ_1}\le \begin{cases}
              \f{c_2C_*}{\sqrt{s_3}}\cL^{-γ}(s_3),&\text{if } σ_1\ge \sqrt{s_3},\\
              \kl{\f{c_3c_2C_*}{\sqrt{s_3}}}^{\f1{β}},&\text{if } σ_1<\sqrt{s_3}.
             \end{cases}
\]
In particular, 
\begin{align}\label{s3oversigma1}
 \kl{\f{\sqrt{s_3}}{σ_1}}^{\f1{θ}}&\le \begin{cases}
                                       \cL^{-\f{γ}{θ}}(s_3) 
(c_2C_*)^{\f1{θ}},&\text{if } 
σ_1\ge \sqrt{s_3},\\
(\sqrt{s_3})^{(1-\f1{β})\f1{θ}} (c_3c_2C_*)^{\f1{βθ}},&\text{if } σ_1<\sqrt{s_3},
                                      \end{cases}\nn\\
&\le c_4C_*^{\f1{θ}}(1+C_*^{\f1{θβ}-\f1{θ}}), 
\end{align}
where we set $c_4:=\max\set{c_2^{\f1{θ}}\cL^{-\f{γ}{θ}}(s_3),(\sqrt{s_3})^{(1-\f1{β})\f1{θ}} 
(c_2c_3)^{\f1{βθ}}}$.

If we moreover use that $C_*=(c_0 K)^{γ} \kl{1+K^{ε_1}}$, we obtain 
\begin{align}\label{cstaretc}
 C_*^{\f1{θ}}(1+C_*^{\f1{θβ}-\f1{θ}}) &\le \kl{c_0K}^{\f{γ}{θ}} \kl{1+K^{ε_1}}^{\f1{θ}} + 
\kl{c_0K}^{\f{γ}{θβ}}\kl{1+K^{ε_1}}^{\f1{θβ}}\nn\\
&\le \max\set{2^{\f1{θ}}c_0^{\f{γ}{θ}},c_0^{\f{γ}{βθ}}2^{\f1{βθ}}} 
\kl{K^{\f{γ}{θ}}\kl{1+K^{\f{ε_1}{θ}}} + 
K^{\f{γ}{θβ}}\kl{1+K^{\f{ε_1}{βθ}}}}\nn\\
&\le 3 \max\set{2^{\f1{θ}}c_0^{\f{γ}{θ}},c_0^{\f{γ}{βθ}}2^{\f1{βθ}}} K^{\f{γ}{θ}} 
\kl{1+K^{\f{γ}{θ}(\f1{β}-1)+\f{ε_1}{βθ}}}\nn\\
&= 3 \max\set{2^{\f1{θ}}c_0^{\f{γ}{θ}},c_0^{\f{γ}{βθ}}2^{\f1{βθ}}} K^{\f{γ}{θ}} 
\kl{1+K^{ε}}
\end{align}

If we merge \eqref{lesqrts3sigma1} with \eqref{s3oversigma1} and \eqref{cstaretc}, 
we thus have 
\begin{align}\label{gni-case-Ib}
 \norm[\Lrn{q}]{φ}&\le
      c_5 \norm[\Lrn r]{∇φ} K^{\nu} (1+K^{ε}) \cL^{-ν}\kl{\norm[\Lrn r]{∇φ}^2},
\end{align}
where $c_5=3c_4 
\max\set{2^{\f1{θ}}c_0^{\f{γ}{θ}},c_0^{\f{γ}{βθ}}2^{\f1{βθ}}}\norm[L^{∞}((0,∞))]{\cL}^{ν}$.

\textbf{Case II: $\norm[\Lrn q]{φ}\le \norm[\Lrn r]{∇φ}$}.\\ 
We let $\cL_\infty=\norm[L^{∞}((0,∞))]{\cL}$ and obtain the following obvious estimate: 
\begin{equation}\label{gni-case-II}
 \norm[\Lrn q]{φ}\le \cL^{ν}_{∞} \cL^{-\nu}(\norm[\Lrn r]{∇φ}^2) \norm[\Lrn r]{∇φ}.
\end{equation}

The combination of \eqref{gni-case-Ia}, \eqref{gni-case-Ib} and \eqref{gni-case-II} yields the 
theorem.
\end{proof}

\textbf{Acknowledgements.} The first author was supported in part by the Slovak
Research and Development Agency under the contract No. APVV-18-0308 and by the VEGA grant
1/0347/18. 

\begin{thebibliography}{10}

\bibitem{abreu_ceccon_montenegro}
E.~Abreu, J.~Ceccon, and M.~Montenegro.
\newblock Extremals for sharp {GNS} inequalities on compact manifolds.
\newblock {\em Ann. Mat. Pura Appl. (4)}, 194(5):1393--1421, 2015.

\bibitem{agueh_masstransport}
M.~Agueh.
\newblock Sharp {G}agliardo-{N}irenberg inequalities and mass transport theory.
\newblock {\em J. Dynam. Differential Equations}, 18(4):1069--1093, 2006.

\bibitem{agueh_plaplacian}
M.~Agueh.
\newblock Sharp {G}agliardo-{N}irenberg inequalities via {$p$}-{L}aplacian type
  equations.
\newblock {\em NoDEA Nonlinear Differential Equations Appl.}, 15(4-5):457--472,
  2008.

\bibitem{badr_gni_on_manifolds}
N.~Badr.
\newblock Gagliardo-{N}irenberg inequalities on manifolds.
\newblock {\em J. Math. Anal. Appl.}, 349(2):493--502, 2009.

\bibitem{brezisbrowder}
H.~Brezis and F.~Browder.
\newblock Partial differential equations in the 20th century.
\newblock {\em Adv. Math.}, 135(1):76--144, 1998.

\bibitem{ceccon_montenegro}
J.~Ceccon and M.~Montenegro.
\newblock Optimal {$L^p$}-{R}iemannian {G}agliardo-{N}irenberg inequalities.
\newblock {\em Math. Z.}, 258(4):851--873, 2008.

\bibitem{cordero_erausquin_etal}
D.~Cordero-Erausquin, B.~Nazaret, and C.~Villani.
\newblock A mass-transportation approach to sharp {S}obolev and
  {G}agliardo-{N}irenberg inequalities.
\newblock {\em Adv. Math.}, 182(2):307--332, 2004.

\bibitem{dao_diaz_nguyen}
N.~A. Dao, J.~I. D\'{\i}az, and Q.-H. Nguyen.
\newblock Generalized {G}agliardo-{N}irenberg inequalities using {L}orentz
  spaces, {BMO}, {H}\"{o}lder spaces and fractional {S}obolev spaces.
\newblock {\em Nonlinear Anal.}, 173:146--153, 2018.

\bibitem{delPino_dolbeault_bestconstants}
M.~Del~Pino and J.~Dolbeault.
\newblock Best constants for {G}agliardo-{N}irenberg inequalities and
  applications to nonlinear diffusions.
\newblock {\em J. Math. Pures Appl. (9)}, 81(9):847--875, 2002.

\bibitem{duoandikoetxea_vega}
J.~Duoandikoetxea and L.~Vega.
\newblock Some weighted {G}agliardo-{N}irenberg inequalities and applications.
\newblock {\em Proc. Amer. Math. Soc.}, 135(9):2795--2802, 2007.

\bibitem{esfahani_anisotropic}
A.~Esfahani.
\newblock Anisotropic {G}agliardo-{N}irenberg inequality with fractional
  derivatives.
\newblock {\em Z. Angew. Math. Phys.}, 66(6):3345--3356, 2015.

\bibitem{filawin}
M.~Fila and M.~Winkler.
\newblock A {G}agliardo-{N}irenberg-type inequality and its applications to
  decay estimates for solutions of a degenerate parabolic equation.
\newblock {\em Adv. Math.}, 357, 2019, https://doi.org/10.1016/j.aim.2019.106823

\bibitem{gagliardo_proprieta}
E.~Gagliardo.
\newblock Propriet\`a di alcune classi di funzioni in pi\`u variabili.
\newblock {\em Ricerche Mat.}, 7:102--137, 1958.

\bibitem{gagliardo_ulteriori_proprieta}
E.~Gagliardo.
\newblock Ulteriori propriet\`a di alcune classi di funzioni in pi\`u
  variabili.
\newblock {\em Ricerche Mat.}, 8:24--51, 1959.

\bibitem{hale_raugel}
J.~K. Hale and G.~Raugel.
\newblock Upper semicontinuity of the attractor for a singularly perturbed
  hyperbolic equation.
\newblock {\em J. Differential Equations}, 73(2):197--214, 1988.

\bibitem{kalamajska_krbec}
A.~Ka{\l}amajska and M.~Krbec.
\newblock Gagliardo-{N}irenberg inequalities in regular {O}rlicz spaces
  involving nonlinear expressions.
\newblock {\em J. Math. Anal. Appl.}, 362(2):460--470, 2010.

\bibitem{kalamajska_pietruskapaluba}
A.~Ka{\l}amajska and K.~Pietruska-Pa{\l}uba.
\newblock Gagliardo-{N}irenberg inequalities in weighted {O}rlicz spaces.
\newblock {\em Studia Math.}, 173(1):49--71, 2006.

\bibitem{kalamajska_pietruskapaluba_withoutdoubling}
A.~Ka{\l}amajska and K.~Pietruska-Pa{\l}uba.
\newblock Gagliardo-{N}irenberg inequalities in weighted {O}rlicz spaces
  equipped with a nonnecessarily doubling measure.
\newblock {\em Bull. Belg. Math. Soc. Simon Stevin}, 15(2):217--235, 2008.

\bibitem{kozono_sato_wadade}
H.~Kozono, T.~Sato, and H.~Wadade.
\newblock Upper bound of the best constant of a {T}rudinger-{M}oser inequality
  and its application to a {G}agliardo-{N}irenberg inequality.
\newblock {\em Indiana Univ. Math. J.}, 55(6):1951--1974, 2006.

\bibitem{kozono_wadade}
H.~Kozono and H.~Wadade.
\newblock Remarks on {G}agliardo-{N}irenberg type inequality with critical
  {S}obolev space and {BMO}.
\newblock {\em Math. Z.}, 259(4):935--950, 2008.

\bibitem{ledoux}
M.~Ledoux.
\newblock On improved {S}obolev embedding theorems.
\newblock {\em Math. Res. Lett.}, 10(5-6):659--669, 2003.

\bibitem{martin_milman}
J.~Martin and M.~Milman.
\newblock Sharp {G}agliardo-{N}irenberg inequalities via symmetrization.
\newblock {\em Math. Res. Lett.}, 14(1):49--62, 2007.

\bibitem{mccormick_etal}
D.~S. McCormick, J.~C. Robinson, and J.~L. Rodrigo.
\newblock Generalised {G}agliardo-{N}irenberg inequalities using weak
  {L}ebesgue spaces and {BMO}.
\newblock {\em Milan J. Math.}, 81(2):265--289, 2013.

\bibitem{morosi_pizzocchero}
C.~Morosi and L.~Pizzocchero.
\newblock On the constants for some fractional {G}agliardo-{N}irenberg and
  {S}obolev inequalities.
\newblock {\em Expo. Math.}, 36(1):32--77, 2018.

\bibitem{nirenberg}
L.~Nirenberg.
\newblock On elliptic partial differential equations.
\newblock {\em Ann. Scuola Norm. Sup. Pisa (3)}, 13:115--162, 1959.

\bibitem{paicu_raugel_rekalo}
M.~Paicu, G.~Raugel, and A.~Rekalo.
\newblock Regularity of the global attractor and finite-dimensional behavior
  for the second grade fluid equations.
\newblock {\em J. Differential Equations}, 252(6):3695--3751, 2012.

\bibitem{raugel_sell}
G.~Raugel and G.~R. Sell.
\newblock Navier-{S}tokes equations on thin {$3$}{D} domains. {I}. {G}lobal
  attractors and global regularity of solutions.
\newblock {\em J. Amer. Math. Soc.}, 6(3):503--568, 1993.

\bibitem{riviere_strzelecki}
T.~Rivi\`ere and P.~Strzelecki.
\newblock A sharp nonlinear {G}agliardo-{N}irenberg-type estimate and
  applications to the regularity of elliptic systems.
\newblock {\em Comm. Partial Differential Equations}, 30(4-6):589--604, 2005.

\bibitem{strzelecki}
P.~Strzelecki.
\newblock Gagliardo-{N}irenberg inequalities with a {BMO} term.
\newblock {\em Bull. London Math. Soc.}, 38(2):294--300, 2006.

\bibitem{weinstein}
M.~I. Weinstein.
\newblock Nonlinear {S}chr\"{o}dinger equations and sharp interpolation
  estimates.
\newblock {\em Comm. Math. Phys.}, 87(4):567--576, 1982/83.

\end{thebibliography}

\end{document}